\documentclass[12pt,reqno]{amsart}

\usepackage{amssymb}

\newtheorem{theorem}{Theorem}[section]
\newtheorem{proposition}[theorem]{Proposition}
\newtheorem{lemma}[theorem]{Lemma}
\newtheorem{corollary}[theorem]{Corollary}

\theoremstyle{definition}

\newtheorem{remark}[theorem]{Remark}

\numberwithin{equation}{section}

\begin{document}

\title[Connections, determinants and the Tyurin parametrization]{Meromorphic
connections, determinant line bundles and the Tyurin parametrization}

\author[Indranil Biswas]{Indranil Biswas}

\address{School of Mathematics, Tata Institute of Fundamental Research,
Homi Bhabha Road, Mumbai 400005}

\email{indranil@math.tifr.res.in}

\author[Jacques Hurtubise]{Jacques Hurtubise}

\address{Department of Mathematics, McGill University, Burnside
Hall, 805 Sherbrooke St. W., Montreal, Que. H3A 2K6, Canada}

\email{jacques.hurtubise@mcgill.ca}



\date{}

\begin{abstract}
We develop a holomorphic equivalence between on one hand the space of pairs (stable 
bundle, flat connection on the bundle) and the ``sheaf of holomorphic connections'' (the sheaf of 
holomorphic splittings of the one-jet sequence) for the determinant (Quillen) line bundle over the 
moduli space of vector bundles on a compact connected Riemann surface. This equivalence is shown to be 
holomorphically symplectic. The equivalences, both holomorphic and symplectic, are rather 
quite general, for example, they extend to other general families of holomorphic bundles and 
holomorphic connections, in particular those arising from ``Tyurin families" of stable 
bundles over the surface. These families generalize the Tyurin parametrization of 
stable vector bundles $E$ over a compact connected Riemann surface, and one can build 
above them spaces of (equivalence classes of) holomorphic connections, which are again symplectic. 
These spaces are also symplectically biholomorphically equivalent to the sheaf of 
holomorphic connections for the determinant bundle over the Tyurin family. The last portion of the 
paper shows how this extends to moduli of framed bundles.
\end{abstract}

\maketitle

\tableofcontents

\section{Introduction}\label{sec0}

We will address in this paper an equivalence which seems to hold in a certain generality both 
in the holomorphic and holomorphically symplectic category between two objects, defined over 
various moduli space of stable bundles on a Riemann surface $X$. The equivalence is somewhat 
surprising, as we do not have a direct map between the two spaces, but instead obtain the 
equivalence by showing that certain cohomology classes are the same in both cases, and indeed 
have the same representative.

Our basic examples of the first set of spaces are moduli spaces of pairs $(E, \nabla)$ with $E$ a stable
bundle and $\nabla$ a holomorphic connection on $E$. Let $X$ be a compact connected Riemann
surface. A stable vector bundle over $X$ of degree zero 
admits holomorphic connections. Once we fix a point $x_0\, \in\, X$, a stable vector bundle 
$E$ on $X$ of degree $d$ and rank $r$ admits logarithmic connections on $X$ nonsingular on 
$X\setminus\{x_0\}$ whose residue at the point $x_0$ is $-\frac{d}{r}\text{Id}_{E_{x_0}}$. Let 
${\mathcal C}$ denote the moduli space of pairs of the form $(E,\, D)$, where $E$ is a stable 
vector bundle on $X$ of rank $r$ and degree $d$, and $D$ is a holomorphic or logarithmic 
connection on $E$ of the above type depending on whether $d$ is zero
or not. This moduli space ${\mathcal C}$ is equipped with a 
natural holomorphic symplectic structure constructed by Goldman and Atiyah--Bott \cite{Go}, 
\cite{AB}. Let $\mathcal M$ denote the moduli space of stable vector bundles on $X$ of rank $r$ 
and degree $d$. The projection ${\mathcal C}\, \longrightarrow\, \mathcal M$, $(E,\, D)\, 
\longmapsto\, E$, has the structure of a holomorphic $T^*\mathcal M$--torsor. 

The more general family in the first set of spaces corresponding to pairs of a vector bundle and 
a holomorphic connection on it is developed from the Tyurin parametrization. This parametrization of 
vector bundles on curves was introduced in \cite{Ty1}, \cite{Ty2}. To explain it briefly, 
fixing a holomorphic vector bundle $V$ of rank $r$ on a compact connected Riemann surface 
$X$, consider all torsion quotients of $V$ of degree $d'$. This way a holomorphic family 
of vector bundles on $X$ of rank $r$ and degree $d\, :=\, \text{degree}(V)- d'$ is obtained,
as the kernels of natural projection to the torsion quotients. Choosing $d'$ correctly gives a 
parametrization of some suitable open set of the moduli space of stable vector bundles
of rank $r$ and degree $d$. This 
Tyurin parametrization has turned out to be very useful; see \cite{Kr}, \cite{KN}, 
\cite{Hu}, \cite{She} and references therein. We now allow more general families by 
letting the degree $d'$ vary. Let ${\mathcal Q}_s$ be a Tyurin parameter space of stable 
vector bundles on $X$ of rank $r$ and degree $d$, built from a fixed vector bundle of
rank $r$ and degree 
$d'$. For a point $z\, \in\, {\mathcal Q}_s$, the corresponding stable vector bundle on 
$X$ will be denoted by ${\mathcal K}^z$. Our space ${\mathcal C}({\mathcal Q}_s)$ will be 
built from the pull-back from $\mathcal M$ of the earlier mentioned
holomorphic fiber bundle ${\mathcal C}$, taking the fiber 
product with $T^*{\mathcal Q}_s$ and quotienting by an equivalence relation. The resulting 
variety ${\mathcal C}({\mathcal Q}_s)$ is a $T^*{\mathcal Q}_s$--torsor and has a
holomorphic symplectic structure (Lemma \ref{lem3}).

Our second set of spaces will be the sheaves of connections on a determinant line bundle.
Given any holomorphic line bundle $L$ on a complex manifold $Z$, let
$\rho\, :\, \text{Conn}(L)\, \longrightarrow\, Z$ be the holomorphic fiber bundle
given by the ``sheaf'' of holomorphic connections on $L$. More precisely,
$\text{Conn}(L)\, \subset\, \text{At}(L)^*\,=\, J^1(L)\otimes L^*$ is the
inverse image, for the natural projection $\text{At}(L)^*\, \longrightarrow\,
{\mathcal O}_Z$, of the image of the section of ${\mathcal O}_Z$ given by the constant
function $1$ on $Z$; here $\text{At}(L)$ denote the Atiyah bundle for $L$ (it
contains ${\mathcal O}_Z$ in a natural way with the quotient being $TZ$). Then $\text{Conn}(L)$ is
a holomorphic torsor over $Z$ for the holomorphic cotangent bundle $T^*Z$, and it is equipped
with a holomorphic symplectic structure given by the curvature of the tautological
holomorphic connection on the holomorphic line
bundle $\rho^*L\, \longrightarrow\, \text{Conn}(L)$.
In our cases the line bundles $L$ will be the natural determinant line bundles over our moduli spaces of
vector bundles.

The main theorems are that the spaces of the first set are equivalent torsors over the 
various moduli of vector bundles to the spaces in the second set; for the natural symplectic 
forms, this equivalence is also a holomorphic symplectic equivalence (Proposition 
\ref{prop1}, Theorem \ref{thm1}, Theorem \ref{thm2}). Note that the Tyurin families 
step outside of the family of stable bundles and, since the Quillen metric is uniformly 
defined, this provides symplectic ``extensions'' of the space ${\mathcal C}$. In some 
sense, the determinant bundle is a much more robust object, and the equivalence of the 
torsors allows us to better understand the symplectic geometry of the ``connection space'' 
${\mathcal C}({\mathcal Q}_s)$. These latter spaces, with their torsor structures, have 
proved extremely useful in understanding the symplectic and Hamiltonian aspects of 
isomonodromic deformations; see \cite{Kr}, \cite{KN}, \cite{Hu}.

For the framed version, fix a nonzero effective divisor $\mathbb S$ on $X$. Let ${\mathcal D}^F$ be the space corresponding
to the triples of the form $(z,\, \sigma,\, D)$, where $z\, \in\, {\mathcal Q}_s$ and $\sigma$ is
a framing on ${\mathcal K}^z$ over $\mathbb S$ while $D$ is a
meromorphic connection on ${\mathcal K}^z$ whose polar part has support contained in $\mathbb S$. This
moduli space ${\mathcal D}^F$ has a holomorphic symplectic structure (Corollary \ref{cor2}).
This holomorphic symplectic structure on ${\mathcal D}^F$ is constructed by identifying
${\mathcal D}^F$ with $\text{Conn}(L')$, where $L'$ is the determinant line bundle on
the moduli space of pairs of the form $(z,\, \sigma)$, where $z\, \in\, {\mathcal Q}_s$
and $\sigma$ is a framing on the vector bundle ${\mathcal K}^z$ over $\mathbb S$ (Theorem \ref{thm4}).

\section{Isomorphism of torsors for the cotangent bundle}

\subsection{Moduli spaces of vector bundles and connections}\label{se2.1}

The holomorphic cotangent
bundle of $Z$ will be denoted by $T^*Z$. The real tangent bundle of $Z$ will
be denoted by $T^{\mathbb R}Z$.

Let $X$ be a compact connected Riemann surface of genus $g$, with $g\, \geq\, 2$.
Let $K_X$ denote the holomorphic cotangent bundle of $X$. The holomorphic tangent
bundle of a complex manifold $Z$ will be denoted by $TZ$.

For fixed integers $r\, \geq \,2$ and $d$, let
\begin{equation}\label{e1}
{\mathcal M}_X(r,d)\, =:\, {\mathcal M}
\end{equation}
denote the moduli space of all stable vector bundles on $X$ of rank $r$ and degree $d$. This
$\mathcal M$ is an irreducible smooth quasiprojective complex variety of dimension
$r^2(g-1)+1$ (see \cite{Ne}). It is projective if and only if $d$ is coprime to $r$.

If $d\, \not=\, 0$, fix a point $x_0\,\in\, X$. Let
\begin{equation}\label{e2}
{\mathcal C}_X(r,d)\, =:\, {\mathcal C}
\end{equation}
denote the moduli space of isomorphism classes of logarithmic connections $(E,\, D)$,
where
\begin{enumerate}
\item $E\, \in\, {\mathcal M}$ (defined in \eqref{e1}), and

\item $D$ is a logarithmic connection on $E$, nonsingular over $X\setminus \{x_0\}$, such
that the residue ${\rm Res}(D,x_0)$ of $D$ at $x_0$ is $-\frac{d}{r}\text{Id}_{E_{x_0}}$.
(See \cite{De} for logarithmic connections and their residues.)
\end{enumerate}

If $d\,=\, 0$, then 
$$
{\mathcal C}_X(r,d)\, =:\, {\mathcal C}
$$
denotes the moduli space of isomorphism classes of holomorphic connections $(E,\, D)$,
where
\begin{enumerate}
\item $E\, \in\, {\mathcal M}$, and

\item $D$ is a holomorphic connection on $E$. (See \cite{At} for holomorphic connections.)
\end{enumerate}
It is known that for every $E\, \in\, {\mathcal M}$, there is a logarithmic/holomorphic
connection $D$ of the above type.

The moduli space $\mathcal C$ is an irreducible smooth quasiprojective complex variety of dimension
$2(r^2(g-1)+1)$ (see \cite{Si1}, \cite{Si2}, \cite{Ni}). It is equipped with an algebraic
symplectic form
\begin{equation}\label{gos}
\Omega_{\mathcal C}\, \in\, H^0({\mathcal C},\, \bigwedge\nolimits^2 T^*{\mathcal C})
\end{equation}
\cite{AB}, \cite{Go}; the symplectic form on $\mathcal C$
constructed in \cite{AB}, \cite{Go} is, a priori, only holomorphic, however
in \cite[p.~331, Theorem 3.2]{Bi} it is shown that this holomorphic symplectic form
coincides with a natural algebraic form on $\mathcal C$.

Let
\begin{equation}\label{e3}
\phi\, :\, {\mathcal C}\, \longrightarrow\, {\mathcal M}
\end{equation}
be the forgetful map $(E,\, D)\, \longmapsto\, E$ that forgets the logarithmic/holomorphic
connection $D$ on $E$.

Given any $(E,\, D)\, \in\, {\mathcal C}$, and any $\theta \, \in\, H^0(X,\,
\text{End}(E)\otimes K_X)$, note that $$(E,\, D+\theta)\, \in\, {\mathcal C}\, .$$ Conversely,
if $(E,\, D')\, \in\, {\mathcal C}$, then we have $$D'-D\, \in\, H^0(X,\,
\text{End}(E)\otimes K_X)\, .$$ On the other hand, $H^0(X,\,
\text{End}(E)\otimes K_X)$ is the fiber $T^*_E {\mathcal M}\,=\, (T_E{\mathcal M})^*$ of
the holomorphic cotangent bundle of ${\mathcal M}$ at the point $E$. The projection $\phi$ in \eqref{e3}
and the above fiberwise action of $T^*{\mathcal M}$ on the moduli space $\mathcal C$ in \eqref{e2}
actually make $\mathcal C$ a complex algebraic torsor over
$\mathcal M$ for the holomorphic cotangent bundle
$T^*{\mathcal M}\, \longrightarrow\, \mathcal M$.

We note that the holomorphic isomorphism classes of the holomorphic
$T^*{\mathcal M}$--torsors on $\mathcal M$ are parametrized by $H^1({\mathcal M},\,
T^*{\mathcal M})$; this aspect is elaborated in Section \ref{se2.3}.

\subsection{The determinant line bundle and a torsor over $\mathcal M$}\label{se2.2}

As before, fix a point $x_0\,\in\, X$.

Let ${\mathcal L}\,\longrightarrow \,\mathcal M$ be the determinant line bundle whose
fiber over any $E\,\in \,\mathcal M$ is
$$(\bigwedge\nolimits^{\rm top}H^0(X,\, E)^*\otimes
\bigwedge\nolimits^{\rm top}H^1(X,\, E))^{\otimes r}\otimes (\bigwedge\nolimits^r E_{x_0})^{\chi}\, ,
$$
where $\chi\, :=\,\chi (E)\,=\, h^0(E) - h^1(E)\,=\, d- r(g-1)$ is the Euler characteristic.
We shall briefly recall the construction of $\mathcal L$.
Given any holomorphic family of stable vector bundles $V$ of rank $r$ and degree $d$ on $X$
$$
V \,\longrightarrow \, X\times T\, \stackrel{p}{\longrightarrow}\, T
$$
parametrized by a complex manifold $T$, we have the holomorphic line bundle
\begin{equation}\label{db}
(\det R^0p_*V)^{\otimes -r}\otimes (\det R^1p_*V)^{\otimes r}
\otimes \det (s_{x_0}^* V)^{\otimes \chi} \,\longrightarrow\, T\, ,
\end{equation}
where $\det {\mathcal W}\,\longrightarrow\, T$ is the determinant line bundle for a coherent analytic sheaf
${\mathcal W}$ on $T$ (see \cite[Ch.~V, \S~6]{Ko} for the construction of determinant bundle) and
$$s_{x_0}\, :\, T\, \longrightarrow\, X\times T\, , \ \ t\, \longmapsto\, (x_0,\, t)$$
is the section of the projection $p$, while $\chi$ as before is $d- r(g-1)\,\in\,
\mathbb Z$. The holomorphic
line bundle on $T$ in \eqref{db} does not change if the vector bundle $V$ is replaced by
$V\otimes p^* L_0$, where $L_0$ is a holomorphic line bundle on $T$. Therefore, the construction
in \eqref{db} produces a holomorphic line bundle
\begin{equation}\label{e4}
{\mathcal L}\, \longrightarrow\, {\mathcal M}
\end{equation}
which is in fact algebraic.

It should be mentioned that the determinant line bundle $\mathcal L$ exists even when there 
is no Poincar\'e bundle over $X\times \mathcal M$. See also \cite{Qu}, \cite{KM} for the 
construction of the determinant line bundle in a general context.

Consider the Atiyah exact sequence
\begin{equation}\label{zeta}
0\, \longrightarrow\,{\mathcal O}_{\mathcal M} \, \longrightarrow\,\text{At}({\mathcal L})
\, \stackrel{\zeta}{\longrightarrow}\, T{\mathcal M} \, \longrightarrow\, 0
\end{equation}
for the holomorphic line bundle $\mathcal L$ in \eqref{e4}, where $\text{At}({\mathcal L})
\,=\, {\mathcal L}\otimes J^1({\mathcal L})^*$ is the Atiyah bundle for $\mathcal L$
with $J^1({\mathcal L})$ being the first jet bundle for $\mathcal L$ (see \cite{At}). Let
\begin{equation}\label{e5}
0\, \longrightarrow\, T^*{\mathcal M}\,=\, \Omega^1_{\mathcal M} \, \longrightarrow\,\text{At}({\mathcal L})^*
\, \stackrel{q}{\longrightarrow}\, {\mathcal O}_{\mathcal M} \, \longrightarrow\, 0
\end{equation}
be the dual sequence of \eqref{zeta}; note that $\text{At}({\mathcal L})^*\,=\,
J^1({\mathcal L})\otimes {\mathcal L}^*$, and $J^1({\mathcal L})$ fits in the short
exact sequence
$$
0\, \longrightarrow\, \Omega^1_{\mathcal M} \, \longrightarrow\, J^1({\mathcal L}) \, \longrightarrow\,
J^0({\mathcal L})\,=\, {\mathcal L} \, \longrightarrow\, 0\, .
$$
The section of ${\mathcal O}_{\mathcal M}$ given
by the constant function $1$ on $\mathcal M$ will be denoted by $1_{\mathcal M}$. Define
\begin{equation}\label{e6}
\text{At}({\mathcal L})^* \, \supset\, q^{-1}(1_{\mathcal M}) \, =:\,
{\rm Conn}({\mathcal L})\, \stackrel{q_0}{\longrightarrow}\, {\mathcal M}\, ,
\end{equation}
where $q$ is the projection in \eqref{e5}, and $q_0$ is the restriction of $q$ to
the subvariety ${\rm Conn}({\mathcal L})\, \subset\,
\text{At}({\mathcal L})^*$. From \eqref{e5} it follows immediately that
${\rm Conn}({\mathcal L})$ is a complex algebraic torsor for the holomorphic cotangent bundle
$T^*{\mathcal M}\, \longrightarrow\, \mathcal M$.

\subsection{Cohomological invariants for holomorphic torsors}\label{se2.3}

Let $Z$ be a complex manifold and $V$ a holomorphic vector bundle over $Z$. The isomorphism 
classes of holomorphic torsors ${\mathbb V}\, \longrightarrow\, Z$ for $V$ are parametrized by 
$H^1(Z,\, V)$. To see this, choose local holomorphic sections of $s_i\, :\, 
U_i\,\longrightarrow\, {\mathbb V}\vert_{U_i}$, where $\{U_i\}_{i\in I}$ is an open covering 
of $Z$. For any ordered pair $(i,\, j)\, \in\, I\times I$, consider $s_i-s_j$ on $U_i\cap 
U_j$, which is in fact a holomorphic section of $V\vert_{U_i\cap U_j}$. This $1$--cocycle 
$\{s_i-s_j\}_{i,j\in I}$ gives the element of $H^1(Z,\, V)$ corresponding to $\mathbb V$.

There is also a Dolbeault type construction of the above cohomology class in
$H^1(Z,\, V)$ associated to the holomorphic $V$--torsor ${\mathbb V}$. For this first note
that since the fibers of ${\mathbb V}\, \longrightarrow\, Z$ are contractible, there are
$C^\infty$ sections of this fiber bundle (however there is a holomorphic section if and only if
the torsor $\mathbb V$ for $V$ is trivial). Take a $C^\infty$ section $s$ of the fiber bundle
${\mathbb V}\, \longrightarrow\, Z$. The obstruction for $s$ to be holomorphic is clearly
the failure
of the differential $ds$ of the map $s$ to intertwine the almost complex structures of
$Z$ and $\mathbb V$. More precisely, consider the homomorphism
\begin{equation}\label{e8}
(ds)'\, :\, T^{\mathbb R}Z\, \longrightarrow\, s^* T^{\mathbb R}{\mathbb V}\, ,\ \
v\, \longmapsto\, \frac{1}{2}(ds(v) + J_{\mathbb V}(ds(J_Z(v))))\, ,
\end{equation}
where $J_Z$ and $J_{\mathbb V}$ are the almost complex structures on $Z$ and $\mathbb V$
respectively, and $$ds\,:\, T^{\mathbb R}Z\, \longrightarrow\, s^*T^{\mathbb R}{\mathbb V}$$ is
the differential of the map $s$. It is straight-forward to check that
\begin{itemize}
\item $(ds)'\,=\, 0$ if and only if the map $s$ is holomorphic,

\item $(ds)'$ is a $C^\infty$ section of $\Omega^{0,1}_Z\otimes V$ over $Z$, and

\item $\overline{\partial}_V ((ds)')\,=\, 0$, so $(ds)'$ defines an element
of the Dolbeault cohomology $H^1(Z,\, V)$.
\end{itemize}
The element of $H^1(Z,\, V)$ defined by $(ds)'$ coincides with the \v{C}ech cohomology
class constructed earlier using locally
defined holomorphic sections of $\mathbb V$. Note that when $V$ is the
holomorphic cotangent bundle $T^*Z$, the above section $(ds)'$ is a
$\overline{\partial}$--closed $(1,1)$--form on $Z$.

The class in $H^1({\mathcal M},\, T^*{\mathcal M})$ corresponding to the $T^*{\mathcal
M}$--torsor ${\rm Conn}({\mathcal L})$ constructed in \eqref{e6}
is $2\pi\sqrt{-1}\cdot c_1({\mathcal L})$, where
$c_1({\mathcal L})$ is the rational first Chern class of the holomorphic
line bundle $\mathcal L$ in \eqref{e4}.

In Section \ref{se2.1} we saw that $\mathcal C$ is a torsor over $\mathcal M$
for $T^* {\mathcal M}$, and in Section \ref{se2.2} we saw that
${\rm Conn}({\mathcal L})$ is a torsor over $\mathcal M$
for $T^*{\mathcal M}$. We shall compare the isomorphism classes of these
two torsors.

First consider the projection $\phi$ in \eqref{e3}. Given any $E\,\in\, \mathcal M$,
by a theorem of Narasimhan and Seshadri, \cite{NS}, there is a unique logarithmic connection
$D_E$ on $E$ such that
\begin{enumerate}
\item $D_E$ is nonsingular on $X\setminus\{x_0\}$,

\item the monodromy of $D_E$ lies in ${\rm U}(r)$, and

\item the residue ${\rm Res}(D_E,x_0)$ of $D_E$ at $x_0$ is
$-\frac{d}{r}\text{Id}_{E_{x_0}}$.
\end{enumerate}
If $d\,=\, 0$, then $D_E$ is a holomorphic connection on $E$ whose monodromy lies
in ${\rm U}(r)$. Therefore, the projection $\phi$ in \eqref{e3} has a canonical section
\begin{equation}\label{e7}
\beta\, :\, {\mathcal M}\, \longrightarrow\, {\mathcal C}
\end{equation}
given by $E\, \longmapsto\, D_E$. This section $\beta$ is $C^\infty$, however it is
not holomorphic.

The moduli space $\mathcal M$ is equipped with a natural K\"ahler form \cite{AB},
\cite{Go}; this K\"ahler form on $\mathcal M$ will be denoted by $\omega_{\mathcal M}$.
We briefly recall the construction of $\omega_{\mathcal M}$. Just as in the
construction of $\beta$ in \eqref{e7}, identify $\mathcal M$ with the equivalence classes of
unitary representations of $\pi_1(X\setminus\{x_0\})$ such that the monodromy around $x_0$ is
$\exp(2\pi\sqrt{-1}d/r)\cdot{\rm Id}$. On the other hand, such a representation space is equipped
with the Goldman symplectic form. This symplectic form coincides with the K\"ahler form
$\omega_{\mathcal M}$. It should be mentioned that
\begin{equation}\label{f12}
\omega_{\mathcal M}\,=\, \beta^*\Omega_{\mathcal C}\, ,
\end{equation}
where $\Omega_{\mathcal C}$ is the holomorphic symplectic form on $\mathcal C$ in \eqref{gos}.

The following lemma is proved in \cite[p.~308, Theorem 2.11]{BR}.

\begin{lemma}\label{lem1}
For the $C^\infty$ section $\beta$ in \eqref{e7}, the corresponding $(1,1)$--form $(d\beta)'$
on $\mathcal M$ constructed in \eqref{e8} coincides with $\omega_{\mathcal M}/2$.
\end{lemma}

\begin{proof}
In \cite[p.~308, Theorem 2.11]{BR} it was proved that the $(1,1)$--form $(d\beta)'$ coincides
with $\frac{1}{2}\beta^*\Omega_{\mathcal C}$, where $\Omega_{\mathcal C}$ is the
symplectic form on $\mathcal C$ in \eqref{gos}. As noted in \eqref{f12}, the pulled back form
$\beta^*\Omega_{\mathcal C}$ coincides with $\omega_{\mathcal M}$.
\end{proof}

Next consider the projection $q_0$ in \eqref{e6}. Quillen in \cite{Qu} using analytic torsion
constructed a Hermitian structure on the holomorphic line bundle ${\mathcal L}$ constructed in
\eqref{e4}; he also computed the curvature of the corresponding Chern connection
on ${\mathcal L}$. Let
\begin{equation}\label{e9}
\gamma\, :\, {\mathcal M}\, \longrightarrow\, {\rm Conn}({\mathcal L})
\end{equation}
be the $C^\infty$ section of the projection $q_0$ in \eqref{e6} given by the Chern connection
associated to the Quillen metric on ${\mathcal L}$.

The following lemma is proved in \cite[p.~320, Theorem 4.20]{BR}.

\begin{lemma}\label{lem2}
For the $C^\infty$ section $\gamma$ in \eqref{e9}, the corresponding $(1,1)$--form $(d\gamma)'$
on $\mathcal M$ constructed in \eqref{e8} coincides with $r\cdot\omega_{\mathcal M}$.
\end{lemma}

It may be clarified that in our case the integer $N$ in \cite[p.~320, Theorem 420]{BR} is $1$.

Let
\begin{equation}\label{e10}
\delta\, :\, {\mathcal C}\times_{\mathcal M} T^*{\mathcal M}\, \longrightarrow\,
{\mathcal C}\ \ \text{ and }\ \ \eta\, :\, {\rm Conn}({\mathcal L})\times_{\mathcal M}
T^*{\mathcal M}\, \longrightarrow\, {\rm Conn}({\mathcal L})
\end{equation}
be the $T^*{\mathcal M}$--torsor structures on $\mathcal C$ and ${\rm Conn}({\mathcal L})$
respectively. Let
\begin{equation}\label{e11}
{\textbf m}\, :\, T^*{\mathcal M}\, \longrightarrow\,T^*{\mathcal M}\, ,\ \ v\, \longmapsto\, 2r\cdot v
\end{equation}
be the multiplication by $2r$.

\begin{proposition}\label{prop1}
There is a unique holomorphic isomorphism
$$
F\, :\, {\mathcal C}\, \longrightarrow\, {\rm Conn}({\mathcal L})
$$
such that
\begin{enumerate}
\item $\phi\, =\, q_0\circ F$, where $\phi$ and $q_0$ are the projections in
\eqref{e3} and \eqref{e6} respectively,

\item $F\circ\beta\,=\, \gamma$, where $\beta$ and $\gamma$ are the sections in
\eqref{e7} and \eqref{e9} respectively, and

\item $F\circ\delta \,=\, \eta \circ (F\times {\textbf m})$, where $\delta$, $\eta$ and
$\textbf m$ are constructed in \eqref{e10} and \eqref{e11}.
\end{enumerate}
\end{proposition}

\begin{proof}
It is straight-forward to check that there is a unique $C^\infty$
diffeomorphism
\begin{equation}\label{ef}
F\, :\, {\mathcal C}\, \longrightarrow\, {\rm Conn}({\mathcal L})
\end{equation}
that satisfies the above three conditions. To prove the proposition
we need to show that this map $F$ in \eqref{ef} is actually holomorphic.

From Lemma \ref{lem1} and Lemma \ref{lem2} it can be deduced that for
point every $y\, \in\, \mathcal M$,
the differential of $F$ takes the almost complex structure on $T_{\beta(y)}\mathcal C$ to the
almost complex structure on $T_{\gamma(y)}{\rm Conn}(\mathcal L)$,
where $\beta$ and $\gamma$ are the sections in
\eqref{e7} and \eqref{e9} respectively. Indeed, this follows from
the constructions of $(d\beta)'$ and $(d\gamma)'$ (see
\eqref{e8}), and the fact that they differ by
multiplication by $2r$ (which follows from Lemma \ref{lem1} and Lemma \ref{lem2}).

Consider the diffeomorphisms
$$
\alpha_1\, : \,T^*{\mathcal M} \,\longrightarrow\, {\mathcal C} \ \ \text{ and }\ \
\alpha_2\, : \,T^*{\mathcal M} \,\longrightarrow\, {\rm Conn}({\mathcal L})
$$
that send any $v \,\in\, T^*_y{\mathcal M}$ to $\beta(y)+v$ and $\gamma(y)+2r\cdot v$ 
respectively. From the three properties of the map $F$ in \eqref{ef} it follows that
\begin{equation}\label{d1}
F\circ \alpha_1 \,=\,\alpha_2\, .
\end{equation}
Let $J_1$ (respectively, $J_2$) denote the almost complex structure on the total 
space of $T^*{\mathcal M}$ obtained by pulling back the almost complex structure on ${\mathcal 
C}$ (respectively, ${\rm Conn}({\mathcal L})$) using this diffeomorphism $\alpha_1$ 
(respectively, $\alpha_2$).

From the above observation that the differential of $F$ takes the almost complex structure 
on $T_{\beta(y)}\mathcal C$ to the almost complex structure on $T_{\gamma(y)}{\rm 
Conn}(\mathcal L)$ for every $y\, \in\, \mathcal M$, it follows that $J_1$ and $J_2$ coincide
for every point $(y,\, 0)\,\in\, T^*{\mathcal M}$, $y\, \in\, \mathcal M$. Also, the restrictions of
$J_1$ and $J_2$ to the fiber $T^*_y{\mathcal M}$ coincide for every $y\, \in\, \mathcal M$.

Let $J_0$ denote the natural almost complex structure on $T^*\mathcal M$ given by the
complex structure on $\mathcal M$. Both the almost complex structures
$J_1$ and $J_2$ on $T^*\mathcal M$ have the property
that they coincide with the tautological almost complex structure $J_0$.
Using these properties of $J_1$ and $J_2$ it follows that $J_1$ coincides
with $J_2$. In view of \eqref{d1}, this implies that the map $F$ is holomorphic. 
\end{proof}

\begin{remark}\label{rem1}
There is an algebro-geometric construction of an isomorphism of the form in 
Proposition \ref{prop1}. To simplify, let us suppose that we are in the case where the degree is $r(g-1)$;
this means that $H^0(X,\,E) \,=\,0\,=\, H^1(X,\, E)$ for all $E$ in a nonempty
Zariski open subset of ${\mathcal M}_X(r,r(g-1))$. The locus where this does not hold is 
the theta divisor $\Theta\, \subset\, {\mathcal M}_X(r,r(g-1))$. There is a natural section
of the determinant line bundle, which vanishes on $\Theta$, so we have a short exact sequence
$$0\,\longrightarrow\, {\mathcal O}_{{\mathcal M}_X(r,r(g-1))}\,\longrightarrow \,
{\mathcal L} \,\longrightarrow \, {\mathcal L}|_\Theta\, \,\longrightarrow \, 0$$
For any $E\, \in\, {\mathcal M}_X(r,r(g-1))\setminus \Theta$,
on $X\times X$, the K\"unneth formula and Serre duality together
give $$H^0(X\times X, \,E\boxtimes (E^*\times K_X))\,=\,0\,=\,
H^1(X\times X,\, E\boxtimes (E^*\otimes K_X))\, ,$$ so from the long exact sequence of cohomologies for
the short exact sequence of sheaves
$$
0\,\longrightarrow\, E\boxtimes (E^*\times K_X)\,\longrightarrow\, E\boxtimes (E^*\otimes K_X) (\Delta)
\,\longrightarrow\,(E\boxtimes (E^*\otimes K_X) (\Delta))\vert_\Delta\,\longrightarrow\, 0\, ,
$$
where $\Delta\, \subset\, X\times X$ is the reduced diagonal and $W(\Delta)\,=\, W\otimes
{\mathcal O}_{X\times X}(\Delta)$ for any holomorphic vector bundle $W$ on
$X\times X$, we obtain an isomorphism
$$H^0(X\times X,\, E\boxtimes (E^*\otimes K_X) (\Delta)) \,= \,
H^0(\Delta, \,(E\boxtimes (E^*\otimes K_X) (\Delta))\vert_\Delta)
\,=\, H^0(X, \,{\rm End}(E)) \, .$$
Now choose a theta-characteristic $K^{1/2}_X$ on $X$, and write $E \,=\,
V\otimes K^{1/2}_X$, so that $V$ is of degree zero. Consequently, we have 
$$H^0(X\times X,\, (V\boxtimes V^*) \otimes (K^{1/2}_X\boxtimes K^{1/2}_X)(\Delta))
\,=\, H^0(X, \,{\rm End}(V)) \, . $$
Let $s_V\,\in\, H^0(X\times X,\, (V\boxtimes V^*) \otimes (K^{1/2}_X\boxtimes K^{1/2}_X)(\Delta))$
be the section corresponding to $\text{Id}_V\, \in\, H^0(X, \,{\rm End}(V))$.
Note that $(K^{1/2}_X\boxtimes K^{1/2}_X)(\Delta)\vert_\Delta = {\mathcal O}_\Delta$,
and the isomorphism extends to $2\Delta$, in a unique way if one imposes anti-symmetry
under involution of $X\times X$. Thus, $s_V$ gives a section of
$H^0(2\Delta,\, V\boxtimes V^*)$ that extends the identity automorphism of $V$ over
$\Delta$. But such a section defines a holomorphic connection on $V$.

Now let $\mathbb U$ be the Zariski open dense subset of the moduli space
${\mathcal M}_X(r,0)$ parametrizing all $V$
such that $H^i(X,\, V\otimes K^{1/2}_X)\,=\, 0$ for $i\,=\, 0,\, 1$. The above construction
of holomorphic connection produces a section of $\mathcal C$ over $\mathbb U$. On the other
hand the pullback of the theta line bundle has a canonical trivialization over $\mathbb U$.
Consequently, we get an isomorphism of two $T^*\mathbb U$--torsors on $\mathbb U$ as
in Proposition \ref{prop1}.

A natural question is whether the above isomorphism of $T^*\mathbb U$--torsors coincides with the
isomorphism constructed in Proposition \ref{prop1}.
\end{remark}

\section{Holomorphic symplectic forms}

Consider the holomorphic line bundle $q^*_0{\mathcal L}\, \longrightarrow\, {\rm 
Conn}({\mathcal L})$, where $q_0$ is the projection in \eqref{e6}. Since ${\rm Conn}({\mathcal 
L})$ is defined by the sheaf of holomorphic connections on $\mathcal L$, there is a 
tautological holomorphic connection $\widehat{\nabla}$ on the pulled back holomorphic line bundle
$q^*_0{\mathcal L}$. To briefly describe $\widehat{\nabla}$, first note that there is a
homomorphism
$$
\beta_0\, :\, q^*_0\text{At}({\mathcal L})\, \longrightarrow\,
q^*_0{\mathcal O}_{\mathcal M}\,=\, {\mathcal O}_{{\rm Conn}({\mathcal L})}
$$
given by the tautological splitting of the short exact sequence of holomorphic vector bundles
$$
0\, \longrightarrow\, q^*_0{\mathcal O}_{\mathcal M}\, \longrightarrow\,q^*_0\text{At}({\mathcal L})
\, \longrightarrow\, q^*_0 T{\mathcal M} \, \longrightarrow\,0
$$
on ${\rm Conn}({\mathcal L})$. On the other hand, there is a tautological projection
$$
\beta'_0\, :\, \text{At}(q^*_0{\mathcal L})\, \longrightarrow\,
q^*_0\text{At}({\mathcal L})
$$
such that the diagram
$$
\begin{matrix}
\text{At}(q^*_0{\mathcal L})& \stackrel{\beta'_0}{\longrightarrow} &
q^*_0\text{At}({\mathcal L})\\
~\,~\, \Big\downarrow \zeta'&& ~\,~\,~\,~\,\Big\downarrow q^*_0\zeta\\
T{\rm Conn}({\mathcal L})& \stackrel{dq_0}{\longrightarrow} &
q^*_0T{\mathcal M}
\end{matrix}
$$
is commutative, where $dq_0$ is the differential of the projection $q_0$ in \eqref{e6}, $\zeta$ is the
projection in \eqref{zeta} and $\zeta'$ is the natural projection of
$\text{At}(q^*_0{\mathcal L})$ to $T{\rm Conn}({\mathcal L}) (as in \eqref{zeta})$. Now the composition
$$
\beta_0\circ\beta'_0\, :\, \text{At}(q^*_0{\mathcal L})\, \longrightarrow\,
{\mathcal O}_{{\rm Conn}({\mathcal L})}
$$
gives a splitting of the Atiyah exact sequence for the holomorphic line bundle $q^*_0{\mathcal L}$.
This splitting $\beta_0\circ\beta'_0$ defines the tautological connection $\widehat{\nabla}$
on $q^*_0{\mathcal L}$.

The curvature $\Omega_{\mathcal L}$ of the above holomorphic connection $\widehat{\nabla}$ is a 
closed holomorphic $2$--form on ${\rm Conn}({\mathcal L})$. This holomorphic $2$--form 
$\Omega_{\mathcal L}$ is symplectic. To see this, choose a local holomorphic trivialization of 
$\mathcal L$ over an open subset $U\, \subset\, \mathcal M$. Using the trivial
connection of a trivial line bundle, the inverse image $q^{-1}_0(U)$ gets 
identified with $T^*U$; this identification sends the zero section of $T^*U$ to the section
of $q^{-1}_0(U)$ given by
the trivial connection on ${\mathcal L}\vert_U$. In terms of this identification, the connection 
$\widehat{\nabla}\vert_{q^{-1}_0(U)}$ becomes the Liouville $1$--form on $T^*U$. Therefore, the 
$2$--form $\Omega_{\mathcal L}\vert_{q^{-1}_0(U)}$ coincides with the exterior derivative of the Liouville
$1$--form, and hence it is 
nondegenerate. So $\Omega_{\mathcal L}$ is a symplectic form as it is closed.

We shall describe some properties of the above symplectic form $\Omega_{\mathcal L}$ on ${\rm Conn}({\mathcal L})$.
Take any $C^\infty$ section
$$
s_0\, :\, {\mathcal M}\, \longrightarrow\, {\rm Conn}({\mathcal L})
$$
of the projection $q_0$ in \eqref{e6}, so $q_0\circ s_0\,=\, \text{Id}_{\mathcal M}$. This $s_0$ defines a $C^\infty$
complex connection on $\mathcal L$; we shall denote this
complex connection by $\nabla^{s_0}$. This connection $\nabla^{s_0}$ coincides with
the pulled back connection $s^*_0\widehat{\nabla}$ after invoking the natural identification
of $s^*_0q^*_0{\mathcal L}$ with $\mathcal L$. This implies that
the curvature of the connection $\nabla^{s_0}$ is
the pulled back form $s^*_0\Omega_{\mathcal L}$.

Note that $s^*_0\Omega_{\mathcal L}$ need not be holomorphic, because the map $s_0$ need not 
be holomorphic. Let $$\theta\,\in\, C^\infty(U,\, T^* U)$$ be a smooth $(1,\, 0)$--form defined 
on an open subset $U\, \subset\, {\mathcal M}$. Then $s_0+\theta$ is a $C^\infty$ section of 
$q_0$ over the open subset $U$, which is constructed using the $T^*{\mathcal M}$--torsor 
structure of ${\rm Conn}({\mathcal L})$. Since the curvature of the connection
$\nabla^{s_0}$ on $\mathcal L$, given by a section $s_0$, is $s^*_0\Omega_{\mathcal L}$, it
follows immediately that
\begin{equation}\label{foch0}
((s_0+\theta)^*\Omega_{\mathcal L})\vert_U\,=\, (s^*_0\Omega_{\mathcal L})\vert_U +d\theta\, .
\end{equation}
The fibers of the projection
$q_0$ are Lagrangian with respect to the symplectic form $\Omega_{\mathcal L}$, because the fibers of
the cotangent bundle $T^*{\mathcal M}$ are Lagrangian with respect to the Liouville symplectic form
on $T^*{\mathcal M}$.

As in \eqref{gos}, let $\Omega_{\mathcal C}$ denote the holomorphic symplectic form on $\mathcal C$.

\begin{theorem}\label{thm1}
For the biholomorphism $F$ in Proposition \ref{prop1},
$$
F^*\Omega_{\mathcal L}\,=\, 2r\cdot \Omega_{\mathcal C}\, .
$$
\end{theorem}

\begin{proof}
We shall first show that the symplectic form $\Omega_{\mathcal C}$ on
$\mathcal C$ is compatible with the $T^*{\mathcal M}$--torsor
structure of $\mathcal C$. The compatibility condition in question says
that for every locally defined holomorphic section
$$
{\mathcal M}\, \supset\, U\, \stackrel{\mathbf{\sigma}}{\longrightarrow}\, {\mathcal C}
$$
of the projection $\phi$ in \eqref{e3}, where $U\, \subset\, {\mathcal M}$ is an open subset, and a holomorphic
$1$--form $\theta$ on $U$,
\begin{equation}\label{foch}
({\mathbf{\sigma}}+\theta)^*\Omega_{\mathcal C}\,=\, {\mathbf{\sigma}}^*\Omega_{\mathcal C}
+d\theta\, .
\end{equation}

The set-up of \cite{AB} will be used for proving \eqref{foch}: we compute on the infinite dimensional space of
connections, and then quotient by the gauge group. Fix a $C^\infty$ complex
vector bundle $V$ on $X$ of rank $r$ and degree $d$. Let
$$
{\mathcal A}^{0,1}\, :=\, C^\infty(X,\, \text{End}(V)\otimes (T^{0,1}X)^*)
\ \ \text{ and } \ \ {\mathcal A}^{1,0}\, :=\, C^\infty(X,\, \text{End}(V)\otimes (T^{1,0}X)^*)
$$
be respectively the spaces of all smooth $(0,\, 1)$--forms and $(1,\, 0)$--forms on $X$
with values in $\text{End}(V)$. Using the nondegenerate pairing on ${\mathcal A}^{0,1}\oplus
{\mathcal A}^{1,0}$
\begin{equation}\label{abp}
(\alpha_1,\, \beta_1)\, (\alpha_2,\, \beta_2))
\,\longmapsto\, \int_X \text{trace}(\alpha_1\wedge\beta_2 - \alpha_2\wedge\beta_1)\, ,
\end{equation}
identify ${\mathcal A}^{0,1}\times{\mathcal A}^{1,0}$ with a subset of
the holomorphic cotangent bundle $(T^{1,0}{\mathcal A}^{0,1})^*$.
Therefore, the restriction to ${\mathcal A}^{0,1}\oplus
{\mathcal A}^{1,0}$ of the Liouville symplectic form on $(T^{1,0}{\mathcal A}^{0,1})^*$
coincides with the one given by the pairing in \eqref{abp}. The two-form on
${\mathcal A}^{0,1}\times {\mathcal A}^{1,0}$ given by the pairing in \eqref{abp} will
be denoted by $\Omega'_{AB}$.

A Dolbeault operator on $V$ is a differential operator
$$
\overline{\partial}_1\, :\, V\, \longrightarrow \, V\otimes (T^{0,1}X)^*
$$
of order one satisfying the Leibniz condition which says that
$$
\overline{\partial}_1(fs)\,=\, f\cdot \overline{\partial}_1(s)+ s\otimes \overline{\partial}f\, ,
$$
where $s$ is any locally defined smooth section of $V$ and $f$ is any locally defined smooth
function on $X$. Let $\mathcal B$ denote the space of all Dolbeault operators on $V$. So
$\mathcal B$ is an affine space for the complex vector space ${\mathcal A}^{0,1}$.
Let $\mathcal G$ denote the space of all differential operators
$$
{\partial}_1\, :\, V\, \longrightarrow \, V\otimes (T^{1,0}X)^*
$$
of order one satisfying the condition that
$$
{\partial}_1(fs)\,=\, f\cdot {\partial}_1(s)+ s\otimes {\partial}f\, ,
$$
where $s$ is any locally defined smooth section of $V$ and $f$ is any locally defined smooth
function. So $\mathcal G$ is an affine space for the complex vector space
${\mathcal A}^{1,0}$ defined above.

Therefore, as before, the pairing in \eqref{abp} produces a $2$--form
on ${\mathcal B}\times {\mathcal G}$; this $2$--form on ${\mathcal B}\times {\mathcal G}$
will be denoted by $\Omega_{AB}$. This $\Omega_{AB}$ actually coincides with the $2$--form
$\Omega'_{AB}$ on ${\mathcal A}^{0,1}\times {\mathcal A}^{1,0}$, once we
identify ${\mathcal B}\times {\mathcal G}$ with ${\mathcal A}^{0,1}\times {\mathcal A}^{1,0}$
by fixing a point of ${\mathcal B}\times {\mathcal G}$.

The symplectic form $\Omega_{\mathcal C}$ on $\mathcal C$ is constructed from the above form
$\Omega_{AB}$ as follows.

For any $(\overline{\partial}_1,\, {\partial}_1)\, \in\, {\mathcal B}\times {\mathcal G}$,
note that ${\partial}_1+\overline{\partial}_1$ is a $C^\infty$ complex connection on the
vector bundle $V$. The curvature of the connection ${\partial}_1+\overline{\partial}_1$ will
be denoted by $({\partial}_1+\overline{\partial}_1)^2$. Define
$$
{\mathcal F}\, :=\, \{(\overline{\partial}_1,\, {\partial}_1)\, \in\,
{\mathcal B}\times {\mathcal G}\,\mid\, ({\partial}_1+\overline{\partial}_1)^2\,=\, 0\}\, .
$$
Restrict $\Omega_{AB}$ to $\mathcal F$. Let $\text{Aut}(V)$ denote the
group of all $C^\infty$ automorphisms of the vector bundle $V$ over the identity map
of $X$. The group $\text{Aut}(V)$ acts on $\mathcal F$ by inducing connections on $V$ from given
ones via automorphism of $V$. The above restricted form $\Omega_{AB}\vert_{\mathcal F}$ descends to a
$2$--form to the quotient under this action. The symplectic manifold
$({\mathcal C},\,\Omega_{\mathcal C})$ is given by this quotient
of $\mathcal F$ and the descended $2$--form on the quotient.

Since $\Omega_{AB}$ is given by the Liouville symplectic form $\Omega'_{AB}$ on 
${\mathcal A}^{0,1}\times {\mathcal A}^{1,0}$, we conclude that the identity in
\eqref{foch} holds.

In view of \eqref{foch0} and \eqref{foch}, the theorem follows from
Lemma \ref{lem1}, Lemma \ref{lem2} and Proposition \ref{prop1}.
\end{proof}

\section{Tyurin parametrization}

Fix a holomorphic vector bundle ${\mathcal E}_0$ on $X$ of rank $r$ and degree
$d+d_0$. Let
\begin{equation}\label{e0}
{\mathcal Q}\,:=\, {\mathcal Q}({\mathcal E}_0, d_0)
\end{equation}
denote the quot scheme parametrizing all torsion quotients of ${\mathcal E}_0$ of
degree $d_0$. On $X\times {\mathcal Q}$, there is a short exact sequence of coherent sheaves
\begin{equation}\label{e13}
0\, \longrightarrow\, {\mathcal K} \, \longrightarrow\,p^*_X {\mathcal E}_0 
\, \longrightarrow\, Q \, \longrightarrow\, 0\, ,
\end{equation}
where $p_X\, :\, X\times {\mathcal Q}\, \longrightarrow\, X$ is the natural projection and $Q$ is
the tautological torsion quotient on $X\times {\mathcal Q}$. We note that $\mathcal K$ is a
holomorphic family of vector bundles of rank $r$ degree $d$ on $X$ parametrized by $\mathcal Q$.
For any $z\, \in\, {\mathcal Q}$, let ${\mathcal K}^z\,:=\, {\mathcal K}\vert_{X\times\{z\}}$ be
the vector bundle on $X$ in this family corresponding to the point $z$. Let
\begin{equation}\label{e12}
{\mathcal Q}_s\,:=\, \{z\, \in\, {\mathcal Q} \, \mid\, {\mathcal K}^z \ \text{ is stable}\}
\, \subset\, \mathcal Q
\end{equation}
be the subset that parametrizes all the stable vector bundles in this family
parametrized by $\mathcal Q$. It is known
that ${\mathcal Q}_s$ is a Zariski open subset of $\mathcal Q$ \cite[p.~635, Theorem~2.8(B)]{Ma}
(see also \cite{Sh}). For a generic choice of ${\mathcal E}_0$
one can ensure that this Zariski open subset ${\mathcal Q}_s$ is non-empty.

We shall construct two holomorphic $T^*{\mathcal Q}_s$--torsors over ${\mathcal Q}_s$. We note that the dimension of 
${\mathcal Q}_s$ does not necessarily match that of the space $\mathcal M$, so it will not be merely a question of 
pulling back our two torsors over $\mathcal M$.

Let
\begin{equation}\label{xi}
\xi\,:\, {\mathcal Q}_s\, \longrightarrow\, {\mathcal M}
\end{equation}
be the classifying morphism for the family of stable vector bundles $\mathcal K$ in
\eqref{e13}. So for any $z\, \in\, {\mathcal Q}_s$, the point of $\mathcal M$
corresponding to the stable vector bundle ${\mathcal K}^z$ is $\xi(z)$. Let
\begin{equation}\label{xi2}
(d\xi)^*\, :\, \xi^*\Omega_{\mathcal M}\,=\, \xi^* T^*{\mathcal M} \, \longrightarrow\, 
T^*{\mathcal Q}_s
\end{equation}
be the homomorphism of cotangent bundles given by the dual of the differential
$d\xi\, :\, T{\mathcal Q}_s \, \longrightarrow\, \xi^* T{\mathcal M}$ of
the map $\xi$ constructed in \eqref{xi}. Let
$$
\xi^*{\mathcal C}\, \longrightarrow\, {\mathcal Q}_s
$$
be the pull-back to ${\mathcal Q}_s$ of the holomorphic fiber bundle $\phi\, :\, {\mathcal C}\,
\longrightarrow\,\mathcal M$ in \eqref{e3}. Consider the action of $\xi^* T^*{\mathcal M}$
on the fiber product $$T^*{\mathcal Q}_s \times_{{\mathcal Q}_s} \xi^*{\mathcal C}$$
under which each $v\, \in\, T^*_{\xi(y)} {\mathcal M}$, $y\, \in\, {\mathcal Q}_s$, sends
every $(w,\, u)\, \in\, T^*_y{\mathcal Q}_s\times \phi^{-1}(\xi(y))$ to
$$(w-(d\xi)^*(v),\, u+v)\, \in\, T^*_y{\mathcal Q}_s\times \phi^{-1}(\xi(y))\, .$$ Let
\begin{equation}\label{e14}
{\mathcal C}({\mathcal Q}_s)\, :=\,
(T^*{\mathcal Q}_s\times_{{\mathcal Q}_s} \xi^*{\mathcal C})/\xi^* T^*{\mathcal M}
\end{equation}
be the quotient for this action of $\xi^* T^*{\mathcal M}$ on $T^*{\mathcal Q}_s
\times_{{\mathcal Q}_s} \xi^*{\mathcal C}$. Let
\begin{equation}\label{e16}
\psi\, :\, {\mathcal C}({\mathcal Q}_s)\, \longrightarrow\,{\mathcal Q}_s
\end{equation}
be the natural projection given by the projections $T^*{\mathcal Q}_s\,\longrightarrow\,
{\mathcal Q}_s$ and $\xi^*{\mathcal C}\,\longrightarrow\,
{\mathcal Q}_s$.

{\sc Remark}: This quotienting appears in the paper \cite{Hu} in a different guise.

The translation action of $T^*{\mathcal Q}_s$
on itself and the trivial action of $T^*{\mathcal Q}_s$ on $\xi^*{\mathcal C}$ together
produce an action of $T^*{\mathcal Q}_s$ on $T^*{\mathcal Q}_s\times_{{\mathcal Q}_s}
\xi^*{\mathcal C}$. This action of $T^*{\mathcal Q}_s$ on $T^*{\mathcal Q}_s\times_{{\mathcal Q}_s}
\xi^*{\mathcal C}$ clearly descends to an action of $T^*{\mathcal Q}_s$ on the
quotient space ${\mathcal C}({\mathcal Q}_s)$ constructed in \eqref{e14}. It is straightforward
to check that this action of $T^*{\mathcal Q}_s$ on ${\mathcal C}({\mathcal Q}_s)$
makes $({\mathcal C}({\mathcal Q}_s),\, \psi)$ in \eqref{e16}
a holomorphic torsor on ${\mathcal Q}_s$ for $T^*{\mathcal Q}_s$.

Note that there is a natural map to ${\mathcal C}({\mathcal Q}_s)$ from the moduli space of
pairs of the form $(z,\, D)$, where
\begin{itemize}
\item $z\, \in\, {\mathcal Q}_s$, and

\item $D$ is a logarithmic connection on ${\mathcal K}^z$ (see \eqref{e12}) nonsingular on
$X\setminus\{x_0\}$ such that the residue of $D$ at $x_0$ is $-\frac{d}{r}\text{Id}_{E_{x_0}}$.
\end{itemize}
This map sends $(z,\, D)$ to the point of $\psi^{-1}(z)$ given by the pair
$(0,\, D)\, \in\, T^*{\mathcal Q}_s\times_{{\mathcal Q}_s} \xi^*{\mathcal C}$
(see \eqref{e14}). It may be mentioned that this map to ${\mathcal C}({\mathcal Q}_s)$
from the moduli space of pairs need not be injective or surjective in general. This map is injective
if the differential $d\xi(z)$ is surjective for all
$z\, \in\, {\mathcal Q}_s$, and it is surjective if $d\xi(z)$ is injective for all
$z\, \in\, {\mathcal Q}_s$.

To construct the second torsor for $T^*{\mathcal Q}_s$, let
$$
\widetilde{f}\,:\, X\times {\mathcal Q}\, \longrightarrow\, {\mathcal Q}
$$
be the projection to the second factor. For the family of holomorphic vector bundles
$\mathcal K$ in \eqref{e13} parametrized by $\mathcal Q$, consider the holomorphic line bundle constructed in \eqref{db}.
Let
\begin{equation}\label{l}
L\,:=\, (\det R^0\widetilde{f}_*{\mathcal K})^{\otimes -r}\otimes (\det R^1\widetilde{f}_*
{\mathcal K})^{\otimes r}
\otimes \det (s_{x_0}^* {\mathcal K})^{\chi} \,\longrightarrow\, {\mathcal Q}
\end{equation}
be this line bundle, where $s_{x_0}\, :\, {\mathcal Q}\,\longrightarrow\,
X\times {\mathcal Q}$ as in \eqref{db} is the section $y\, \longmapsto\, (x_0,\, y)$.
Construct
$$
\widetilde{\varphi}\, :\, {\rm Conn}(L)\, \longrightarrow\, {\mathcal Q}
$$
as in \eqref{e6}, so ${\rm Conn}(L)\, \subset\, \text{At}(L)^*$, where $\text{At}(L)
\, \longrightarrow\, {\mathcal Q}$
is the Atiyah bundle for $L$. We note that $({\rm Conn}(L),\,
\widetilde{\varphi})$ is a torsor over $\mathcal Q$ for $T^*{\mathcal Q}$.
Let
\begin{equation}\label{e15}
{\rm Conn}(L)_s\, :=\, \widetilde{\varphi}^{-1}({\mathcal Q}_s) \, \subset\, {\rm Conn}(L)
\end{equation}
be the Zariski open subset, where ${\mathcal Q}_s$ is the Zariski open subset in \eqref{e12}. Let
\begin{equation}\label{e19}
\varphi \, :\, {\rm Conn}(L)_s\, \longrightarrow\, {\mathcal Q}_s
\end{equation}
be the restriction of the map $\widetilde\varphi$ to ${\rm Conn}(L)_s$. Consequently,
$({\rm Conn}(L)_s,\, {\varphi})$ is a torsor over ${\mathcal Q}_s$ for $T^*{\mathcal Q}_s$.

Let
\begin{equation}\label{e17}
\widetilde{\delta}\, :\, {\mathcal C}({\mathcal Q}_s)\times_{{\mathcal Q}_s}
T^*{\mathcal Q}_s\, \longrightarrow\,
{\mathcal C}({\mathcal Q}_s)\ \ \text{ and }\ \ \widetilde{\eta}\, :\,
{\rm Conn}(L)_s\times_{{\mathcal Q}_s}
T^*{\mathcal Q}_s\, \longrightarrow\, {\rm Conn}(L)_s
\end{equation}
be the $T^*{\mathcal Q}_s$--torsor structures on ${\mathcal C}({\mathcal Q}_s)$ and ${\rm Conn}(L)_s$
constructed in \eqref{e16} and \eqref{e19} respectively. Let
\begin{equation}\label{e18}
{\textbf m}_1\, :\, T^*{\mathcal Q}_s\, \longrightarrow\,T^*{\mathcal Q}_s\, ,\ \ v\, \longmapsto\, 2r\cdot v
\end{equation}
be the multiplication by $2r$.

\begin{theorem}\label{thm2}
There is a canonical biholomorphic map
$$
\Phi\, :\, {\mathcal C}({\mathcal Q}_s)\, \longrightarrow\, {\rm Conn}(L)_s
$$
such that
\begin{enumerate}
\item $\varphi\circ\Phi\,=\, \psi$, where $\psi$ and $\varphi$ are the projections in
\eqref{e16} and \eqref{e19} respectively, and

\item $\Phi\circ\widetilde{\delta} \,=\, \widetilde{\eta} \circ (\Phi\times {\textbf m}_1)$,
where $\widetilde\delta$, $\widetilde\eta$ and
${\textbf m}_1$ are constructed in \eqref{e17} and \eqref{e18}.
\end{enumerate}
\end{theorem}

\begin{proof}
From the constructions of the map $\xi$ in \eqref{xi} and the line bundles $\mathcal L$ and 
$L$ (in \eqref{e4} and \eqref{l}) it follows that the line bundle $\xi^*{\mathcal L}$ is 
holomorphically identified with $L\vert_{{\mathcal Q}_s}$. In view of this identification of 
$L\vert_{{\mathcal Q}_s}$ with $\xi^*{\mathcal L}$, we conclude that the fiber bundle ${\rm
Conn}(L)_s$ in \eqref{e15} is constructed from ${\rm Conn}({\mathcal L})$ (see
\eqref{e6}) in the following way.

Consider the holomorphic fiber bundle $q_0\, :\, {\rm Conn}({\mathcal L})\, \longrightarrow\, {\mathcal M}$
in \eqref{e6}. Let
$$
\xi^*{\rm Conn}({\mathcal L})\, \longrightarrow\, {\mathcal Q}_s
$$
be the pull-back of it by the map $\xi$ in \eqref{xi}. Next consider the homomorphism
$(d\xi)^*$ in \eqref{xi2}. The holomorphic vector bundle
$\xi^* T^*{\mathcal M}$ acts
on the fiber product $$T^* {{\mathcal Q}_s}\times_{{\mathcal Q}_s} \xi^*{\rm Conn}({\mathcal L})$$
as follows: for every $y\, \in\, {\mathcal Q}_s$ and
every $v\, \in\, T^*_{\xi(y)} {\mathcal M}$, the action of $v$ sends any $(w,\, u)\, \in\,
T^*_y {{\mathcal Q}_s}\times q_0^{-1}(\xi(y))$ to $$(w-(d\xi)^*(v),\, u+v)
\, \in\, T^*_y {{\mathcal Q}_s}\times q_0^{-1}(\xi(y))\, .$$ Let
\begin{equation}\label{z}
{\mathcal Z}\,:=\,
(T^*{{\mathcal Q}_s}\times_{{\mathcal Q}_s} \xi^*{\rm Conn}({\mathcal L}))/\xi^* T^*{\mathcal M}
\end{equation}
be the quotient for this action. The projections $T^*{{\mathcal Q}_s}\, \longrightarrow\,
{{\mathcal Q}_s}$ and $\xi^*{\rm Conn}({\mathcal L}), \longrightarrow\,
{{\mathcal Q}_s}$ together produce a projection of ${\mathcal Z}$ to ${\mathcal Q}_s$.
The translation action of $T^*{{\mathcal Q}_s}$
on itself and the trivial action of $T^*{{\mathcal Q}_s}$ on $\xi^*{\rm Conn}({\mathcal L})$ together
produce an action of $T^*{{\mathcal Q}_s}$ on
$T^*{{\mathcal Q}_s}\times_{{\mathcal Q}_s} \xi^*{\rm Conn}({\mathcal L})$. This action
in turn produces an action of $T^*{{\mathcal Q}_s}$ on the quotient space ${\mathcal Z}$ constructed
in \eqref{z}. This action of $T^*{{\mathcal Q}_s}$ on $\mathcal Z$
makes $\mathcal Z$ a holomorphic torsor on ${\mathcal Q}_s$ for $T^*{{\mathcal Q}_s}$. It should
be emphasized that this $T^*{{\mathcal Q}_s}$--torsor
$\mathcal Z$ is holomorphically identified, in a natural way, with the $T^*{{\mathcal Q}_s}$--torsor
${\rm Conn}(L)_s$ constructed in \eqref{e15}.

The biholomorphism $\Phi$ in the statement of the theorem is constructed by comparing the above description
of ${\rm Conn}(L)_s$ with the construction of ${\mathcal C}({\mathcal Q}_s)$ in \eqref{e14}. To see this,
consider the map
$$
\widetilde{F}\, :\, T^*{{\mathcal Q}_s}\times_{{\mathcal Q}_s} \xi^*{\mathcal C}\, \longrightarrow\,
T^*{{\mathcal Q}_s}\times_{{\mathcal Q}_s} \xi^*{\rm Conn}({\mathcal L})\, , \ \ (a,\, b)\,
\longmapsto\, (a,\, \xi^* F(b))\, ,
$$
where $F$ is the biholomorphism in Proposition \ref{prop1}; note that
$F$ induces a map of fiber bundles over ${\mathcal Q}_s$
$$\xi^* F\, :\, \xi^*{\mathcal C}\, \longrightarrow\,
\xi^*{\rm Conn}({\mathcal L})$$ which is uniquely determined by the following
commutative diagram:
$$
\begin{matrix}
\xi^*{\mathcal C} & \stackrel{\xi^* F}{\longrightarrow} &\xi^*{\rm Conn}({\mathcal L})\\
\Big\downarrow && \Big\downarrow\\
{\mathcal C} & \stackrel{F}{\longrightarrow} & {\rm Conn}({\mathcal L})
\end{matrix}
$$
(the vertical maps are the canonical ones).
This map $\widetilde{F}$ descends to a map between the quotient spaces
$$
\Phi\, :\, {\mathcal C}({\mathcal Q}_s)\, \longrightarrow\, {\rm Conn}(L)_s\,=\, {\mathcal Z}
$$
in \eqref{z} and \eqref{e14}. From the properties of $F$ in Proposition \ref{prop1} it follows that
$\Phi$ satisfies the two conditions in the theorem.
\end{proof}

The pulled back holomorphic line bundle $\varphi^*L\, \longrightarrow\, {\rm Conn}(L)_s$,
where $\varphi$ is the projection in \eqref{e19},
has a tautological holomorphic connection; this tautological holomorphic connection on
$\varphi^*L$ will be denoted by $D_{\varphi^*L}$. The curvature of $D_{\varphi^*L}$, which
will be denoted by $\widetilde\Theta$, is a holomorphic symplectic form on ${\rm Conn}(L)_s$.
Consider the biholomorphism $\Phi$ in Theorem \ref{thm2}. Let
\begin{equation}\label{sf}
\Theta\, :=\, \frac{1}{2r} \Phi^* \widetilde{\Theta}
\end{equation}
be the holomorphic symplectic form on ${\mathcal C}({\mathcal Q}_s)$.

The above construction is summarized in the following lemma.

\begin{lemma}\label{lem3}
The complex manifold ${\mathcal C}({\mathcal Q}_s)$ is equipped with a natural holomorphic symplectic
form $\Theta$ constructed in \eqref{sf}. The fibers of the projection $\psi$ in \eqref{e16}
are Lagrangian with respect to $\Theta$. The form $\Theta$ is compatible with the
$T^*{\mathcal Q}_s$--torsor structure on ${\mathcal C}({\mathcal Q}_s)$ in the following
way: If $s\, :\, U\, \longrightarrow\, {\mathcal C}({\mathcal Q}_s)$ is a holomorphic 
section of the fibration $\psi$ over an open subset $U\, \subset\, {\mathcal Q}_s$,
and $\omega$ is a holomorphic $1$--form on $U$, then for the holomorphic section
$$
s_\omega\, :\, U\, \longrightarrow\, {\mathcal C}({\mathcal Q}_s)\, , \ \ z\, \longmapsto\,
s(z) + \omega(z)\, ,
$$
the equation
$$
s^*_\omega \Theta \,=\, s^*\Theta + d\omega
$$
holds.
\end{lemma}

\begin{proof}
For the symplectic form $\widetilde{\Theta}$ on ${\rm Conn}(L)_s$, the fibers
of $\varphi$ (see \eqref{e19}) are Lagrangian. 

If $s\, :\, U\, \longrightarrow\, {\rm Conn}(L)_s$ is a holomorphic 
section of the fibration $\varphi$ over an open subset $U\, \subset\, {\mathcal Q}_s$,
and $\omega$ is a holomorphic $1$--form on $U$, then for the holomorphic section
$$
s_\omega\, :\, U\, \longrightarrow\, {\rm Conn}(L)_s\, , \ \ z\, \longmapsto\,
s(z) + \omega(z)\, ,
$$
the equation
$$
s^*_\omega \widetilde{\Theta} \,=\, s^*\widetilde{\Theta} + d\omega
$$
holds; see \eqref{foch0}.

Therefore, from Theorem \ref{thm2} we conclude that $\Theta$
has the properties stated in the lemma.
\end{proof}

\begin{remark}\label{rem2}
Note that the $T^*{\mathcal Q}_s$--torsor ${\mathcal C}({\mathcal Q}_s)$ does not
have any natural extension to a $T^*{\mathcal Q}$--torsor over the larger
variety ${\mathcal Q}$ in \eqref{e12}. Indeed, for some $z\, \in\, {\mathcal Q}\setminus
{\mathcal Q}_s$ the vector bundle ${\mathcal K}^z$ in \eqref{e12} may not have any
logarithmic/holomorphic connection satisfying the residue condition in the definition
on $\mathcal C$ (see \eqref{e2}). However, the other $T^*{\mathcal Q}_s$--torsor, namely
${\rm Conn}(L)_s$, has a canonical extension to a $T^*{\mathcal Q}$--torsor over the larger
variety ${\mathcal Q}$. Indeed, ${\rm Conn}(L)$ is the canonical extension of
${\rm Conn}(L)_s$ (see \eqref{e15}). Note that the symplectic structure on ${\rm Conn}(L)_s$
also extends along this extension.
\end{remark}

\section{Framed bundles and meromorphic connections}

\subsection{Framed bundles}

Fix a nonzero effective divisor
\begin{equation}\label{bs}
{\mathbb S}\,=\, \sum_{i=1}^n n_ix_i
\end{equation}
on $X$; so $x_i\, \in\, X$, $n\,\geq\, 1$ and $n_i\,\geq\, 1$
for all $1\,\leq\, i\,\leq\,n$. For
notational convenience, $V\otimes {\mathcal O}_X(-{\mathbb S})$ and $V\otimes {\mathcal O}_X({\mathbb S})$,
where $V$ is any coherent analytic sheaf on $X$, will be denoted by $V(-{\mathbb S})$ and $V({\mathbb S})$
respectively.

A framed vector bundle is a holomorphic vector bundle $E$ on $X$ together with
an isomorphism of ${\mathcal O}_X$--modules $\sigma \,:\, E\vert_{\mathbb S} \, \longrightarrow\,
{\mathcal O}^{\oplus r}_{\mathbb S}$, where $r\,=\, \text{rank}(E)$. The space of infinitesimal
deformations of a framed bundle $(E, \, \sigma)$ are parametrized by $H^1(X,\, \text{End}(E)
(-{\mathbb S}))$. Consider the short exact sequence of coherent analytic sheaves
$$
0\, \longrightarrow\, \text{End}(E) (-{\mathbb S})\, \longrightarrow\, \text{End}(E)
\, \longrightarrow\, \text{End}(E)\vert_{\mathbb S} \, \longrightarrow\, 0
$$
on $X$. Let
$$
\, \longrightarrow\, H^0(X,\, \text{End}(E)\vert_{\mathbb S}) \, \stackrel{h_1}{\longrightarrow}\,
H^1(X,\, \text{End}(E)(-{\mathbb S}))\, \stackrel{h_2}{\longrightarrow}\, H^1(X,\,
\text{End}(E)) \, \longrightarrow\, 0
$$
be the corresponding long exact sequence of cohomologies. The homomorphism $h_1$ in this long exact
sequence corresponds to deforming the framing $\sigma$ keeping the vector bundle $E$ fixed, while the
other homomorphism $h_2$ is the forgetful map that sends an infinitesimal
deformation of $(E, \, \sigma)$ to the infinitesimal deformation of $E$ given by it by simply
forgetting the framing. Note that by Serre duality,
\begin{equation}\label{e22}
H^1(X,\, \text{End}(E)(-{\mathbb S}))^* \,=\, H^0(X,\, \text{End}(E)\otimes K_X({\mathbb S}))\, .
\end{equation}

A meromorphic connection on $E$ is a holomorphic differential operator of order one
$$
D\, :\, E \, \longrightarrow\, E\otimes K_X({\mathbb S})
$$
that satisfies the Leibniz identity which says that
$$
D(fs)\,=\, f\cdot D(s) +s \otimes \partial f\, ,
$$
where $s$ is any locally defined holomorphic section of $E$ and $f$ is any locally defined
holomorphic function on $X$.

A \textit{framed meromorphic connection} is a triple of the form $(E,\, \sigma, \, D)$, where 
$(E,\, \sigma)$ is a framed bundle and $D$ is a meromorphic connection on $E$.

Let 
\begin{equation}\label{e20}
{\mathcal N}_X(r,d)\, =:\, {\mathcal N}
\end{equation}
be the moduli space of all isomorphism classes framed vector bundle $(E,\, \sigma)$ of rank $r$ and
degree $d$ such that the underlying vector bundle $E$ is stable. Let
\begin{equation}\label{fl}
{\mathcal D}_X(r,d)\, =:\, {\mathcal D}
\end{equation}
denote the moduli space of isomorphism classes of framed meromorphic connections
$(E,\, \sigma, \, D)$ such that
\begin{enumerate}
\item $E$ is a stable vector bundle of rank $r$ and degree $d$,

\item $\sigma$ is a framing on $E$ over $\mathbb S$, and

\item $D$ is a meromorphic connection on $E$ whose polar part has support contained in $\mathbb S$.
\end{enumerate}
Let
\begin{equation}\label{e21}
\widehat{\phi}\, :\, {\mathcal D}\, \longrightarrow\, {\mathcal N}\, , \ \
(E,\, \sigma, \, D)\, \longmapsto\, (E,\, \sigma)
\end{equation}
be the forgetful map that simply forgets the meromorphic connection. Since
the divisor $\mathbb S$ is nonzero, it can be shown that any stable vector bundle admits a
meromorphic connection. Indeed, if $E$ is a stable vector bundle on $X$ of rank $r$ and
degree $d$, and $x_1\, \in\, {\mathbb D}$, then $E$ admits a logarithmic connection
nonsingular over $X\setminus \{x_1\}$ whose monodromy is unitary and its residue at
$x_1$ is $-\frac{d}{r}\text{Id}_{E_{x_1}}$ \cite{NS}.

The space of all meromorphic connections on $E$ is an affine space for the vector space 
$H^0(X,\, \text{End}(E)\otimes K_X({\mathbb S}))$. Hence using \eqref{e22} it follows that 
$({\mathcal D},\, \widehat{\phi})$ is a holomorphic $T^*{\mathcal N}$--torsor over ${\mathcal N}$,
where ${\mathcal D}$ and $\widehat{\phi}$ are constructed in \eqref{fl} and \eqref{e21} respectively.

Let
\begin{equation}\label{vp}
\varpi\, : \, {\mathcal N}\, \longrightarrow\, {\mathcal M}
\end{equation}
be the projection
defined by $(E,\, \sigma)\, \longmapsto\, E$. Consider the line bundle 
${\mathcal L}$ in \eqref{e4} constructed by setting the base point $x_0$ to be
the point $x_1$ in \eqref{bs}. Let
$$
\varpi^* {\mathcal L}\, \longrightarrow\, {\mathcal N}
$$
be its pullback to ${\mathcal N}$ by the map in \eqref{vp}. Construct
\begin{equation}\label{q1}
q_1\, :\, {\rm Conn}(\varpi^* {\mathcal L})\, \longrightarrow\, {\mathcal N}
\end{equation}
as in \eqref{e6} from the Atiyah bundle $\text{At}(\varpi^* {\mathcal L})$. We note that
$({\rm Conn}(\varpi^* {\mathcal L}),\, q_1)$ is a torsor over $\mathcal N$ for $T^*{\mathcal N}$.

Let
\begin{equation}\label{e23}
\widehat{\delta}\, :\, {\mathcal D}\times_{\mathcal N}
T^*{\mathcal N}\, \longrightarrow\, {\mathcal D}\ \ \text{ and }\ \ \widetilde{\eta}\, :\,
{\rm Conn}(\varpi^* {\mathcal L})\times_{\mathcal N}
T^*{\mathcal N} \, \longrightarrow\, {\rm Conn}(\varpi^* {\mathcal L})
\end{equation}
be the $T^*{\mathcal N}$--torsor structures on $\mathcal D$ and ${\rm Conn}(\varpi^* {\mathcal L})$
respectively. Let
\begin{equation}\label{e24}
{\textbf m}_2\, :\, T^*{\mathcal N}\, \longrightarrow\,T^*{\mathcal N}\, ,\ \ v\, \longmapsto\, 2r\cdot v
\end{equation}
be the multiplication by $2r$.

\begin{theorem}\label{thm3}
There is a canonical biholomorphic map
$$
\Psi\, :\, {\mathcal D}\, \longrightarrow\, {\rm Conn}(\varpi^* {\mathcal L})
$$
such that
\begin{enumerate}
\item $q_1\circ\Psi\,=\, \widehat{\phi}$, where $\widehat{\phi}$ and $q_1$ are the
projections in \eqref{e21} and \eqref{q1} respectively, and

\item $\Psi\circ\widehat{\delta} \,=\, \widehat{\eta} \circ (\Psi\times {\textbf m}_2)$,
where $\widehat\delta$, $\widehat\eta$ and
${\textbf m}_2$ are constructed in \eqref{e23} and \eqref{e24}.
\end{enumerate}
\end{theorem}

\begin{proof}
Let $(d\varpi)^*\, : \, \varpi^* T^* {\mathcal M}\, \longrightarrow\, T^*{\mathcal N}$
be the dual of the differential
$$d\varpi\, :\, T{\mathcal N}\, \longrightarrow\,\varpi^* T{\mathcal M}$$
of the map $\varpi$ in \eqref{vp}. Using it, the pullback,
to ${\mathcal N}$, of a $T^* {\mathcal M}$--torsor on ${\mathcal M}$ produces a
$T^*{\mathcal N}$--torsor on ${\mathcal N}$. To give more details of this construction,
let $\mathcal T$ be a $T^* {\mathcal M}$--torsor on ${\mathcal M}$. Consider
the fiber product
$$
(T^* {\mathcal N})\times_{\mathcal N} \varpi^*{\mathcal T}\, \longrightarrow\, {\mathcal N}\, .
$$
Now $\varpi^* T^* {\mathcal M}$ acts on it as follows: for any $z \,\in\, {\mathcal N}$ and
$w\,\in\, T^*_{\varpi(z)}{\mathcal M}$, the action of $w$ sends $(a,\, b)\,\in\, (T^*_z {\mathcal N})
\times {\mathcal T}_{\varpi(z)}$ to $(a-(d\varpi)^*(w),\, b+w)\,\in\,
(T^*_z {\mathcal N})\times {\mathcal T}_{\varpi(z)}$. The quotient
$$
((T^* {\mathcal N})\times_{\mathcal N} \varpi^*{\mathcal T})/\varpi^* T^* {\mathcal M}
\, \longrightarrow\, \mathcal N
$$
is a $T^*{\mathcal N}$--torsor.

Now, ${\mathcal D}$ is identified with the $T^*{\mathcal N}$--torsor
on ${\mathcal N}$ given by the $T^*{\mathcal M}$--torsor ${\mathcal C}$ on
${\mathcal M}$ in \eqref{e3}, while ${\rm Conn}(\varpi^* {\mathcal L})$
is identified with the $T^*{\mathcal N}$--torsor on ${\mathcal N}$ given by the
$T^*{\mathcal M}$--torsor ${\rm Conn}({\mathcal L})$ on ${\mathcal M}$ in
\eqref{e6}. Consequently, the biholomorphism $F$ in Proposition \ref{prop1}
produces the isomorphism $\Psi$ in the statement of the theorem.
\end{proof}

We recall that the pulled back holomorphic line bundle 
$$
q^*_1\varpi^* {\mathcal L}\, \longrightarrow\, {\rm Conn}(\varpi^* {\mathcal L})\, ,
$$
where $q_1$ and $\varpi$ are the maps in \eqref{q1} and \eqref{vp} respectively,
has a tautological holomorphic connection whose curvature is a holomorphic symplectic
form on ${\rm Conn}(\varpi^* {\mathcal L})$. Let $\widetilde{\Theta}_{\mathcal N}$ denote
this holomorphic symplectic form on ${\rm Conn}(\varpi^* {\mathcal L})$.

Theorem \ref{thm3} gives the following:

\begin{corollary}\label{cor1}
The pulled back form $$\Theta_{\mathcal N}\, := \,\frac{1}{2r}\Psi^*\widetilde{\Theta}_{\mathcal N}$$
is a holomorphic symplectic structure on ${\mathcal D}$. The fibers of the projection $\widehat{\phi}$ \eqref{e21}
are Lagrangian with respect to this symplectic form $\Theta_{\mathcal N}$. The form
$\Theta_{\mathcal N}$ is compatible with the
$T^*{\mathcal N}$--torsor structure on ${\mathcal D}$ in the following
way: If $s\, :\, U\, \longrightarrow\, {\mathcal D}$ is any holomorphic section of the
fibration $\widehat{\phi}$ over an open subset $U\, \subset\, {\mathcal N}$, and
$\omega$ is any holomorphic $1$--form on $U$, then for the holomorphic section
$$
s_\omega\, :\, U\, \longrightarrow\, {\mathcal D}\, , \ \ z\, \longmapsto\,
s(z) + \omega(z)\, ,
$$
the equation
$$
s^*_\omega \Theta_{\mathcal N} \,=\, s^*\Theta_{\mathcal N} + d\omega
$$
holds.
\end{corollary}

\begin{proof}
Since $\widetilde{\Theta}_{\mathcal N}$ has the above two properties, it follows from
Theorem \ref{thm3} that $\Theta_{\mathcal N}$ also has these two properties.
\end{proof}

\section{Tyurin parametrization with framings and meromorphic connections}

Consider ${\mathcal Q}_s$ constructed in \eqref{e12}. Let ${\mathcal Q}^F_s$ be the moduli space
of pairs $(z,\, \sigma)$, where
\begin{itemize}
\item $z\, \in\, {\mathcal Q}_s$, and

\item $\sigma$ is a framing, over $\mathbb S$, on the holomorphic vector bundle
${\mathcal K}^z$ (see \eqref{e12}).
\end{itemize}
Let
\begin{equation}\label{vpp}
\varpi'\, :\, {\mathcal Q}^F_s\, \longrightarrow\, {\mathcal Q}_s\, , \ \
(z,\, \sigma)\, \longmapsto\, z
\end{equation}
be the projection.

We shall construct two $T^*{\mathcal Q}^F_s$--torsors over ${\mathcal Q}^F_s$.

To construct the first $T^*{\mathcal Q}^F_s$--torsor, note that there is a natural morphism
\begin{equation}\label{e25}
\varpi^F\,:\, {\mathcal Q}^F_s\, \longrightarrow\, {\mathcal N}\, ,\ \
(z,\, \sigma)\, \longmapsto\, ({\mathcal K}^z,\, \sigma)\, ,
\end{equation}
so $\varpi\circ\varpi^F\,=\,\xi\circ\varpi'$, where $\varpi$ and $\xi$ are constructed in 
\eqref{vp} and \eqref{xi} respectively. Let $$(d\varpi^F)^*\,:\, (\varpi^F)^* T^*{\mathcal N}\, 
\longrightarrow\, T^*{\mathcal Q}^F_s$$ be the dual of the differential
$$
d\varpi^F\, :\, T{\mathcal Q}^F_s\, \longrightarrow\, (\varpi^F)^*T{\mathcal N}
$$
of the map $\varpi^F$ in \eqref{e25}. Consider the fiber product
$$
(T^*{\mathcal Q}^F_s)\times_{{\mathcal Q}^F_s} ((\varpi^F)^*{\mathcal D})\,\longrightarrow\,
{\mathcal Q}^F_s\, ,
$$
where $\widehat{\phi}\, :\, {\mathcal D}\,\longrightarrow\, {\mathcal N}$ is the $T^*{\mathcal
N}$--torsor in \eqref{e21}. Now $(\varpi^F)^* T^*{\mathcal N}$ acts on it as follows. For any
$x\, \in\, {\mathcal Q}^F_s$, the action of $w\,\in\, T^*_{\varpi^F(x)} {\mathcal N}$ on
$T^*_x{\mathcal Q}^F_s\times \widehat{\phi}^{-1} (\varpi^F(x))$ sends any
$(v,\, D)\, \in\, T^*_x{\mathcal Q}^F_s\times \widehat{\phi}^{-1} (\varpi^F(x))$ to
$(v- (d\varpi^F)^*(w),\, D+w)$. Let
\begin{equation}\label{e26}
{\mathcal D}^F \,:=\, ((T^*{\mathcal Q}^F_s)\times_{{\mathcal Q}^F_s} ((\varpi^F)^*{\mathcal D}))/
((\varpi^F)^* T^*{\mathcal N})
\end{equation}
be the corresponding quotient. Let
\begin{equation}\label{e28}
\widehat{\phi}^F\, :\, {\mathcal D}^F \, \longrightarrow\, {\mathcal Q}^F_s
\end{equation}
be the natural map given by the projections $T^*{\mathcal Q}^F_s\, \longrightarrow\,
{\mathcal Q}^F_s$ and $(\varpi^F)^*{\mathcal D}\, \longrightarrow\,
{\mathcal Q}^F_s$.

The translation action of $T^*{\mathcal Q}^F_s$ on itself
and the trivial action of $T^*{\mathcal Q}^F_s$ on $(\varpi^F)^*{\mathcal D}$ together produce
an action of $T^*{\mathcal Q}^F_s$ on $(T^*{\mathcal Q}^F_s)\times_{{\mathcal Q}^F_s}
((\varpi^F)^*{\mathcal D})$. This action descends to an action of $T^*{\mathcal Q}^F_s$ on the
quotient ${\mathcal D}^F$ in \eqref{e26}. Now $({\mathcal D}^F,\,
\widehat{\phi}^F)$ gets the structure of a $T^*{\mathcal Q}^F_s$--torsor on ${\mathcal Q}^F_s$
using this action of $T^*{\mathcal Q}^F_s$ on ${\mathcal D}^F$.

We note that there is a natural map to ${\mathcal D}^F$ from the moduli space triples of the form
$(z,\, \sigma, \, D)$, where
\begin{itemize}
\item $z\, \in\, {\mathcal Q}_s$,

\item $\sigma$ is a framing over $\mathbb S$ on the vector bundle ${\mathcal K}^z$ in \eqref{e12},
and

\item $D$ is a meromorphic connection on ${\mathcal K}^z$.
\end{itemize}
More precisely, any triple $(z,\, \sigma, \, D)$ of the above form is sent to the point of 
$(\widehat{\phi}^F)^{-1}(z,\, \sigma)$ (the map $\widehat{\phi}^F$ is defined in \eqref{e28}) 
corresponding to $(0,\, D)\,\in\, (T^*{\mathcal Q}^F_s)\times_{{\mathcal Q}^F_s} 
((\varpi^F)^*{\mathcal D})$. This map to ${\mathcal D}^F$ from the moduli space of triples need 
not be injective or surjective in general. This map is injective if the differential 
$d\varpi^F(z)$ is surjective for every $z\, \in\, {\mathcal Q}^F_s$, and it is surjective if 
$d\varpi^F(z)$ is injective for every $z\, \in\, {\mathcal Q}^F_s$.

To construct the second $T^*{\mathcal Q}^F_s$--torsor, first consider the holomorphic line bundle
$$(\varpi\circ\varpi^F)^*{\mathcal L}\, \longrightarrow\, {\mathcal Q}^F_s\, ,$$
where $\varpi$ is the projection in \eqref{vp}, and $\mathcal L$ is the holomorphic line
bundle constructed in \eqref{e4}. We note that $(\varpi\circ\varpi^F)^*{\mathcal L}$
coincides with the determinant line bundle over ${\mathcal Q}^F_s$ for the family of vector
bundles $(\varpi')^*{\mathcal K}$ over $X$ parametrized by ${\mathcal Q}^F_s$, where $\mathcal K$
and $\varpi'$ are constructed in \eqref{e13} and \eqref{vpp} respectively.
Construct the holomorphic fiber bundle
\begin{equation}\label{e27}
q^F_1\, :\, {\rm Conn}((\varpi\circ\varpi^F)^*{\mathcal L})\, \longrightarrow\, {\mathcal Q}^F_s
\end{equation}
using the Atiyah bundle $\text{At}((\varpi\circ\varpi^F)^*{\mathcal L})$ that
corresponds to the sheaf of holomorphic connections on $(\varpi\circ\varpi^F)^*{\mathcal L}$.
It is a $T^*{\mathcal Q}^F_s$--torsor over ${\mathcal Q}^F_s$. As noted before,
${\rm Conn}((\varpi\circ\varpi^F)^*{\mathcal L})$ is equipped with a holomorphic
symplectic structure given by
the curvature of the tautological holomorphic connection on the line bundle
$(\varpi\circ\varpi^F\circ q^F_1)^*{\mathcal L}$. Let
\begin{equation}\label{e31}
\widetilde{\Theta}_{{\mathcal Q}^F_s}\, \in\, H^0({\rm Conn}((\varpi\circ\varpi^F)^*{\mathcal L}),
\, \bigwedge\nolimits^2 T^*{\rm Conn}((\varpi\circ\varpi^F)^*{\mathcal L}))
\end{equation}
be this holomorphic symplectic form on ${\rm Conn}((\varpi\circ\varpi^F)^*{\mathcal L})$.

Let
\begin{equation}\label{e29}
\widehat{\delta}^F\, :\, {\mathcal D}^F\times_{{\mathcal Q}^F_s}
T^*{\mathcal Q}^F_s \, \longrightarrow\, {\mathcal D}^F
\end{equation}
and
\begin{equation}\label{e32}
\widehat{\eta}^F\, :\,
{\rm Conn}((\varpi\circ\varpi^F)^*{\mathcal L})\times_{{\mathcal Q}^F_s}
T^*{\mathcal Q}^F_s \, \longrightarrow\, {\rm Conn}((\varpi\circ\varpi^F)^*{\mathcal L})
\end{equation}
be the $T^*{\mathcal Q}^F_s$--torsor structures on ${\mathcal D}^F$ and ${\rm Conn}
((\varpi\circ\varpi^F)^*{\mathcal L})$ constructed in \eqref{e28} and \eqref{e27} respectively. Let
\begin{equation}\label{e30}
{\textbf m}^F\, :\, T^*{\mathcal Q}^F_s\, \longrightarrow\,T^*{\mathcal Q}^F_s\, ,\ \
v\, \longmapsto\, 2r\cdot v
\end{equation}
be the multiplication by $2r$.

Now we have the following analog of Theorem \ref{thm3}.

\begin{theorem}\label{thm4}
There is a canonical biholomorphic map of fiber bundles
$$
\Psi^F\, :\, {\mathcal D}^F\, \longrightarrow\, {\rm Conn}((\varpi\circ\varpi^F)^* {\mathcal L})
$$
such that
\begin{enumerate}
\item $q^F_1\circ\Psi^F\,=\, \widehat{\phi}^F$, where $\widehat{\phi}^F$ and $q^F_1$ are the
projections in \eqref{e28} and \eqref{e27} respectively, and

\item $\Psi^F\circ\widehat{\delta}^F \,=\, \widehat{\eta}^F \circ (\Psi\times {\textbf m}^F)$,
where $\widehat{\delta}^F$, $\widehat{\eta}^F$ and
${\textbf m}^F$ are constructed in \eqref{e29}, \eqref{e32} and \eqref{e30} respectively.
\end{enumerate}
\end{theorem}

\begin{proof}
A proof of Theorem \ref{thm4} can be constructed from the proof of Theorem
\ref{thm3}. We omit the details.
\end{proof}

An analogue of Remark \ref{rem2} persists in this set-up with framings. To elaborate,
let ${\mathcal Q}^F$ be the moduli space
of pairs $(z,\, \sigma)$, where
\begin{itemize}
\item $z\, \in\, {\mathcal Q}$, and

\item $\sigma$ is a framing, over $\mathbb S$, on the holomorphic vector bundle
${\mathcal K}^z$ (see \eqref{e12}).
\end{itemize}
So ${\mathcal Q}^F_s$ is a Zariski open subset of ${\mathcal Q}^F$ (see \eqref{e12}).
The $T^*{\mathcal Q}^F_s$--torsor ${\mathcal D}^F$ over ${\mathcal Q}^F_s$ does not have
natural extension to a $T^*{\mathcal Q}^F$--torsor over ${\mathcal Q}^F$. But the
$T^*{\mathcal Q}^F_s$--torsor ${\rm Conn}((\varpi\circ\varpi^F)^* {\mathcal L})$
over ${\mathcal Q}^F_s$ has a natural extension to a $T^*{\mathcal Q}^F$--torsor over
${\mathcal Q}^F$, because the determinant line bundle $(\varpi\circ\varpi^F)^* {\mathcal L}$
over ${\mathcal Q}^F_s$ has a natural extension to ${\mathcal Q}^F$.

Theorem \ref{thm4} gives the following.

\begin{corollary}\label{cor2}
For the biholomorphic map $\Psi^F$ in Theorem \ref{thm4}, the pulled back form
$$\Theta_{{\mathcal D}^F}\, :=\, \frac{1}{2r}(\Psi^F)^*
\widetilde{\Theta}_{{\mathcal Q}^F_s}$$ (see \eqref{e31}) defines a holomorphic symplectic
structure on ${\mathcal D}^F$. The fibers of the projection $\widehat{\phi}^F$ \eqref{e28}
are Lagrangian with respect to this symplectic form $\Theta_{{\mathcal D}^F}$. The form
$\Theta_{{\mathcal D}^F}$ is compatible with the
$T^*{\mathcal Q}^F_s$--torsor structure on ${\mathcal D}^F$ in the following
way: If $s\, :\, U\, \longrightarrow\, {\mathcal D}^F$ is any holomorphic
section of the fibration $\widehat{\phi}^F$ over an open subset $U\, \subset\, {\mathcal Q}^F_s$,
and $\omega$ is a holomorphic $1$--form on $U$, then for the holomorphic section
$$
s_\omega\, :\, U\, \longrightarrow\, {\mathcal D}^F\, , \ \ z\, \longmapsto\,
s(z) + \omega(z)\, ,
$$
the equation
$$
s^*_\omega \Theta_{{\mathcal D}^F} \,=\, s^*\Theta_{{\mathcal D}^F} + d\omega
$$
holds.
\end{corollary}

\section*{Acknowledgements}

We thank Tony Pantev for useful discussions. The first-named author thanks
Centre de Recherches Math\'ematiques, Montreal, for hospitality. He is partially supported
by a J. C. Bose Fellowship.

%%%%%%%%%%%%%%%%%%%%%%%%%%%%%%%%%%%%%%%%%%%%%%%%%%%%%%%%%%%%%%%%%%%%% 

\end{document}